\documentclass[preprint,11pt]{elsarticle}
\usepackage{fullpage}
\usepackage{amsmath,amsthm,amssymb,graphicx,float,epsf}

\journal{Applied Mathematics and Computation}

\def\RR{\mathbb R}

\newtheorem{Th}{Theorem}[section]
\numberwithin{equation}{section}

\newtheorem{uess}[Th]{Lemma}
\newtheorem{guess}[Th]{Theorem}
\newtheorem{remark}[Th]{Remark}
\newtheorem{corol}[Th]{Corollary}
\newtheorem{example}[Th]{Example}
\newtheorem{definition}[Th]{Definition}

\begin{document}  

\begin{frontmatter}

\title{Boundedness and Persistence of Delay Differential Equations with Mixed Nonlinearity}

\author[label1]{Leonid Berezansky}
\author[label2]{Elena Braverman}
\address[label1]{Dept. of Math.,
Ben-Gurion University of the Negev,
Beer-Sheva 84105, Israel}
\address[label2]{Dept. of Math. and Stats., University of
Calgary,2500 University Drive N.W., Calgary, AB, Canada T2N 1N4; e-mail
maelena@ucalgary.ca, phone 1-(403)-220-3956, fax 1-(403)--282-5150 (corresponding author)}


\begin{abstract}
For a nonlinear equation with several variable delays
$$
\dot{x}(t)=\sum_{k=1}^m f_k(t, x(h_1(t)),\dots,x(h_l(t)))-g(t,x(t)),
$$
where the functions $f_k$ increase in some variables and decrease in the others,
we obtain conditions when a positive solution exists on $[0, \infty)$, as well as
explore boundedness and persistence of solutions. Finally, we present
sufficient conditions when a solution is unbounded. Examples include the
Mackey-Glass equation with non-monotone feedback and two variable delays;
its solutions can be neither persistent nor bounded, unlike the well studied case
when these two delays coincide. 
\end{abstract}

\begin{keyword} 
nonlinear delay differential equations \sep a global positive solution \sep 
persistent, permanent and unbounded solutions \sep population dynamics models \sep Mackey-Glass equation 

\noindent
{\bf AMS Subject Classification:} 34K25, 34K60, 92D25, 34K23
\end{keyword}

\end{frontmatter}

\section{Introduction}

Many mathematical models of population dynamics can be written in the form
of a scalar equation
\begin{equation}\label{1}
\dot{x}(t)=f(x(t-\tau))-x(t),
\end{equation}
where $f$ is a nonnegative continuous function describing reproduction or recruitment, $\tau$ is a positive number describing delay. Usually 
these models have a unique positive equilibrium
$K$, and  there is a well-developed theory on the global stability of the positive equilibrium
of (\ref{1}). This theory was applied to many well-known models described by
Eq.~(\ref{1}) such as Nicholson's blowflies delay equation and Mackey-Glass equations.

Eq.~(\ref{1}) can be extended to the case when both the delay and the intrinsic growth rate are variable
\begin{equation}\label{2}
\dot{x}(t)=r(t) \left[ f(x(h(t))) - x(t) \right],
\end{equation}
where $h(t)\leq t$ and $r(t)>0$ are Lebesgue measurable. Global stability 
results for Eq.~(\ref{2}) with applications
to population dynamics can  be found in \cite{FO,Gopalsamy,GT,ILT,ILT2,IvMammad,Lenbury,LizPituk} and references therein,
see also \cite{review_nichol,BBI1,BBI2,MG3}.
 
Another generalization of (\ref{1}) is the model with several production terms and nonlinear mortality
\begin{equation}\label{3}
\dot{x}(t)=\sum_{k=1}^m f_k(t, x(h_1(t)),\dots,x(h_l(t)))-g(t,x(t)),
\end{equation}
where $f_k, g$ are nonnegative continuous functions. This equation with some applications
was studied, for example, in \cite{BBDS,NA2009,BBI1,GHM,M}. 
 
For all mentioned above equations usual assumptions are the following:
the function $f_k$ is either monotone or unimodal, $g(t,u)$ is monotone increasing in $u$, 
there is only one delay involved in $f_k$, and a positive equilibrium is unique. 
However, it is possible to consider more general models,
for example, the modified Nicholson equation
\begin{equation}\label{4}
\dot{x}(t)=\sum_{k=1}^m a_k(t) x(h_k(t))e^{-\lambda_k x(g_k(t))}-b(t)x(t), \quad t \geq 0,
\end{equation}
and the modified
Mackey-Glass type equation 
\begin{equation}\label{5}
\dot{x}(t)=\sum_{k=1}^m \frac{a_k(t)x(h_k(t))}{1+x^{n_k}(p_k(t))}-\left(b(t)-\frac{c(t)}{1+ x^n(t)}\right) x(t),
\quad t \geq 0.
\end{equation}

There are also many generalizations of Eqs. (\ref{1})-(\ref{5})
to the case of distributed delays and integro-differential equations \cite{JMAA2014,Kinz2006,Zhuk2012,Kuang,Yuan}.

Let us illustrate the idea that the presence of several delays instead of one delay can create a new type of dynamics.
As Example~\ref{exampledis1} illustrates, an equation which was stable for coinciding delays can become 
unstable, once the two delays are different. 

\begin{example}
\label{exampledis1}
Consider the modified Mackey-Glass equation with two delays
\begin{equation}
\label{ex1dis1}
\dot{x}(t)= \frac{2x(h(t))}{1+x^2(g(t))}-x(t), \quad t \geq 0.
\end{equation}
The unique positive equilibrium is $x=1$, the function $f(x)=2x/(1+x^2)$ is increasing on $[0,1]$, 
so any positive solution of the equation
\begin{equation}
\label{ex1dis2}
\dot{x}(t)= \frac{2x(h(t))}{1+x^2(h(t))}-x(t), \quad t \geq 0
\end{equation}
satisfies $\lim\limits_{t\to\infty} x(t)=1$, see, for example, \cite{MG3,Zhuk2012}.
Consider (\ref{ex1dis1}) with piecewise constant arguments $h(\cdot)$, $g(\cdot)$. Denote
$a=\ln(59/24)\approx 0.8994836$ and $b=\ln(134/15)\approx 2.1897896$ and let
$$
\varphi(t)=6.4-5.9 e^{-(t+a+b)},~t \in [-a-b,-b],~~
\varphi(t)=\frac{1}{17}+\frac{67}{17} e^{-(t+b)},~t \in [-b,0],
$$
then $\varphi(-a-b)=0.5$, $\varphi(-b)=4$, $\varphi(0)=\frac{1}{17}+\frac{67}{17}\frac{15}{134}=0.5$.
Assume for $n=0,1,2, \dots$
$$
h(t)= \left\{  \begin{array}{ll} \left[  \frac{t}{a+b} \right]-b,& t\in [n(a+b),n(a+b)+a), \vspace{2mm} \\
\left[  \frac{t}{a+b} \right]-a-b, & t\in [n(a+b)+a,(n+1)(a+b)),  \end{array} \right. 
$$
where $[t]$ is the integer part of $t$,
$$
g(t) = 
\left\{  \begin{array}{ll} \left[  \frac{t}{a+b} \right]-a-b,& t\in [n(a+b),n(a+b)+a), \vspace{2mm} \\
\left[  \frac{t}{a+b} \right]-b, & t\in [n(a+b)+a,(n+1)(a+b)).  \end{array} \right. 
$$
Then the solution is $(a+b)$-periodic, the equation is
$ \dot{x}(t) = \frac{32}{5}-x(t)$ on $[n(a+b),n(a+b)+a)$, $x(n(a+b))=\frac{1}{2}$  and 
$ \dot{x}(t) = \frac{1}{17}-x(t)$ on $[n(a+b)+a,(n+1)(a+b))$, $x(n(a+b)+a)=4$.
Thus, with two delays, the equilibrium $K=1$ of Eq.~(\ref{ex1dis1}) is not globally
asymptotically stable, unlike (\ref{ex1dis2}).
\end{example}

As Example~\ref{exampledis1} illustrates, an equation which was stable for the coinciding delays can have 
oscillating solutions with a constant amplitude which do not tend to the positive equilibrium.
According to Example~\ref{MG_is_bad}, two different delays can lead not only to sustainable oscillations but also  
to unbounded solutions.
 
The purpose of the present paper is to consider a general nonlinear delay equation which
includes (\ref{4}), (\ref{5}) as particular cases and study the following properties of
these equations: existence and uniqueness of a positive global solution,
persistence, permanence, as well as existence of unbounded solutions.
To the best of our knowledge, equations with such mixed types of nonlinearities 
have not been studied before.

Compared to most of the previous publications, we consider two modifications: the production function
is a sum of several functions, and each $f_k$ involves several delays. The situation when several (sometimes
incomparable) delays are included, is quite common, for example, transmission and translation delays in gene regulatory
systems. Motivated by this,  we apply the general results to some well-known
population dynamics equations.

The paper is organized as follows. After introducing some relevant assumptions and definitions in Section 2,
we justify existence of a global positive solution in Section 3. Section 4 deals with sufficient conditions when
all positive solutions are bounded. In Section 5, we investigate persistence of solutions and also consider their 
permanence. Section 6 explores positive unbounded solutions, and Section 7 involves brief discussion.

\section{Preliminaries}

\begin{definition}\label{def1}
We will say that $f(t,u_1,\dots,u_l)$ is {\bf a Caratheodory function}
if in its domain it is continuous in $u_1,\dots,u_l$
for almost all $t$ and is locally essentially bounded in $t$ for any
$u_1,\dots,u_l$.

The function $f(t,u_1,\dots,u_l)$ is {\bf a locally Lipschitz function}
if for any interval $[a,b]$ there exist positive constants $\alpha_k([a,b])$
such that 
$$
|f(t,u_1,\dots,u_l)-f(t,v_1,\dots,v_l)|\leq \sum_{k=1}^l \alpha_k([a,b])|u_k-v_k|,~ 
u_k,v_k\in  [a,b],~k=1, \dots, l,~t \geq 0.
$$
\end{definition}

In this paper we consider the scalar nonlinear equation with several delays
\begin{equation}\label{6}
\dot{x}(t)=\sum_{k=1}^m f_k(t, x(h_1(t)),\dots,x(h_l(t)))-g(t,x(t)), \quad t \geq 0
\end{equation}
under the following conditions:

(a1) $f_k:[0,\infty)\times \RR^{l}\rightarrow [0,\infty), g:[0,\infty)\times \RR\rightarrow [0,\infty)$ are Caratheodory and locally Lipschitz functions, $f_k(t,0,\dots,0)=0, g(t,0)=0$;

(a2) $h_j, j=1\dots,l,$ are Lebesgue measurable functions, $h_j(t)\leq t, ~ \lim_{t\rightarrow\infty} h_j(t)=\infty$.

Together  with Eq.~(\ref{6}) consider an initial condition
\begin{equation}\label{7}
x(t)=\varphi(t),~~ t\leq 0,
\end{equation} 
where

(a3) $\varphi:(-\infty,0]\rightarrow {\mathbb R}$ 
is a nonnegative Borel measurable bounded function, $ \varphi(0)>0$.

\begin{definition}\label{def2}
The solution of problem~(\ref{6}),(\ref{7}) is an absolutely continuous on $[0,\infty)$ function satisfying (\ref{6})
almost everywhere  for $t\geq 0$ and condition (\ref{7}) for $t\leq 0$.
\end{definition}

Instead of the initial point $t=0$ we can consider any initial point $t=t_0>0$.
In Definition~\ref{def2} the interval $[0,\infty)$ can be substituted by the maximum interval $(0,c)$, with 
$c>0$, or $(t_0,c)$, $c>t_0$ where the solution 
exists. However, in the present paper we only consider the case when a global positive solution exists on $[0,\infty)$.
Sufficient conditions for existence of a positive solution on $[0,\infty)$ are discussed in the next section.

\section{Existence of a Positive Solution on $[0,\infty)$}
 
Let us first justify that if (a1)-(a3) are satisfied then the positive local solution of  (\ref{6}),  
(\ref{7}) exists and is unique.

Denote by ${\bf L}^2([t_0,t_1])$ the space of Lebesgue measurable real-valued 
functions $x(t)$ such that \\ $\displaystyle Q=\int_{t_0}^{t_1} (x(t))^2~dt<\infty$, with the usual norm $\| x \|_{{\bf 
L}^2([t_0,t_1])}=\sqrt{Q}$, by ${\bf C}([t_0,t_1])$  the space of continuous on
$[t_0,t_1]$ functions with the $\sup$-norm.

The following result from the book of Corduneanu
\cite[Theorem 4.5, p. 95]{Cordun} will be applied. We recall that an operator $N$ is
{\em causal} (or {\em Volterra}) if for any two functions $x$ and $y$
and each $t$ the fact that $x(s)=y(s)$, $s \leq t$, implies
$(Nx)(s)=(Ny)(s)$, $s \leq t$.

\begin{uess}
\label{lemma1} \cite{Cordun}
Consider the equation
\begin{equation}
\label{cord}
y^{\prime}(t)=({\cal L}y)(t)+({\cal N}y)(t), ~~t\in [t_0,t_1],~ y(t_0)=y_0,
\end{equation}
where ${\cal L}: {\bf C}([t_0,t_1]) \to {\bf L}^2([t_0,t_1])$ is a linear bounded causal operator,
${\cal N}: {\bf C}([t_0,t_1]) \to {\bf L}^2([t_0,t_1])$ is a nonlinear causal 
operator which satisfies
$$
\| {\cal N}x - {\cal N}y \|_{{\bf L}^2([t_0,t_1])} \leq \lambda \| x-y \|_{{\bf C}([t_0,t_1])}
$$
for $\lambda$ sufficiently small. Then there exists a unique absolutely
continuous on $[t_0,t_1]$ solution of (\ref{cord}).
\end{uess}

Let us note that in Lemma~\ref{lemma1}, ${\cal L}$ and ${\cal N}$ are causal operators and thus can include delays.
They are defined on ${\bf C}([t_0,t_1])$ which corresponds to delay equations with the zero initial function
for $t<t_0$. For an arbitrary initial function and $t_0=0$, in the proof of Theorem~\ref{theorem1} we reduce 
the problem to the zero initial function  and $t \geq 0$ only.

\begin{guess}
\label{theorem1}
Suppose (a1)-(a3) hold. Then there exists a unique local positive solution of
(\ref{6}),  (\ref{7}).
\end{guess}
\begin{proof} 
In order to reduce (\ref{6}),  (\ref{7}) to the equation 
which will be considered for $t \geq 0$, we rewrite this problem as 
\begin{equation}\label{8} 
\dot{x}(t)=\sum_{k=1}^m f_k(t, x^{h_1}(t)+\varphi_{h_1}(t), \dots, x^{h_l}(t)+\varphi_{h_l}(t))-g(t,x(t)), 
\end{equation}
where
$$
x^{h_j}(t)=\left\{\begin{array}{ll}
x(h_j(t)),& h_j(t)> 0,\\
0,& h_j(t) \leq  0,\end{array}\right.,~~
\varphi_{h_j}(t)=\left\{\begin{array}{ll}
\varphi(h_j(t)),& h_j(t)\leq 0,\\
0,& h_j(t)> 0. \end{array}\right.
$$
We can consider (\ref{8}) for $t \geq 0$ only (which corresponds to the zero initial condition).

Denote $({\cal L}x)(t)\equiv 0$, 
$\displaystyle
({\cal N}x)(t)=\sum_{k=1}^m f_k(t, x^{h_1}(t)+\varphi_{h_1}(t),\dots, x^{h_l}(t)+\varphi_{h_l}(t))-g(t,x(t)).
$
We have for any $t_0>0$
\begin{eqnarray*}
\| {\cal N}x - {\cal N}y \|_{{\bf L}^2([0,t_0])} & \leq & 
 \left\|\sum_{k=1}^m f_k(\cdot, x^{h_1}(\cdot)+\varphi_{h_1}(\cdot),\dots, x^{h_l}(\cdot)+\varphi_{h_l}(\cdot))
\right. \\ & & \left. 
-\sum_{k=1}^m f_k(\cdot, y^{h_1}(\cdot)
+\varphi_{h_1}(\cdot),\dots, y^{h_l}(\cdot)+\varphi_{h_l}(\cdot))\right\|_{{\bf L}^2([0,t_0])}
\\ & &
+\|g(\cdot,x(\cdot))-g(\cdot,y(\cdot))\|_{{\bf L}^2([0,t_0])}
\\
&  \leq & \sum_{k=1}^m \sum_{j=1}^l \alpha_k^j\|x^{h_j}(\cdot)-y^{h_j}(\cdot)\|_{{\bf L}^2([0,t_0])}
+\beta \|x-y\|_{{\bf L}^2([0,t_0])},
\end{eqnarray*}
where $\alpha_k^j=\alpha_k^j([0,t_0])$, $\beta=\beta([0,t_0])$  are local Lipschitz constants for $f_k$ and 
$g$, respectively.

Then
$$
\| {\cal N}x - {\cal N}y \|_{{\bf L}^2([0,t_0])} \leq
\left(\sum_{k=1}^m \sum_{j=1}^l \alpha_k^j+\beta\right) \sqrt{ t_0} \|x-y\|_{{\bf C}^2([0,t_0])}.
$$
Hence if $t_0$ is sufficiently small, the constant
$\lambda=(\sum_{k=1}^m \sum_{j=1}^l \alpha_k^j+\beta) \sqrt{t_0} $
is also sufficiently small. Thus by Lemma~\ref{lemma1}  there exists a unique
solution of problem (\ref{6}),  (\ref{7}) on $[0,t_0]$. Since $x(0)>0$,
for small $t_0$ this solution is positive, which concludes the proof.
\end{proof}

\begin{guess}\label{theorem2} 
Suppose conditions (a1)-(a3) are satisfied
and at least one of the following assumptions holds:

$(a4_1)$  for any $[a,b]$ there exists 
$\varepsilon>0$ such that $h_j(t)\leq b-\varepsilon$ for $t \in [a,b]$, $j=1,\dots,l$; 

$(a4_2)$  $\displaystyle 0\leq f_k(t, u_1,\dots,u_l)\leq \sum_{j=1}^l a_{kj}(t)u_j+b_k(t)$,
where $a_{kj}, b_k$ are locally integrable functions;

$(a4_3)$ for $x$ sufficiently large $\displaystyle g(t,x)-\sum_{k=1}^m f_k(t, u_1,\dots,u_l)\geq a_x>0, x\geq u_j, j=1,\dots,l$.

Then problem (\ref{6}),  (\ref{7}) has a unique positive global solution on $[0,\infty)$.
\end{guess}
\begin{proof}
By Theorem~\ref{theorem1} there exists a unique local positive solution of this problem. 
Suppose $[0,c)$ is a maximum interval of existence for this solution.
Since $\dot{x}(t)\geq -g(t,x(t))$, 
$x(0)>0$ and $g(t,0)=0$, we have $x(t)>z(t) \equiv 0$, $t\in[0,c)$,
where $z$ is a solution of $\dot{z}(t)=-g(t,z(t))$, $z(0)=0$,
and the solution $z$ is unique due to the local Lipschitz condition for $g$ as a part of (a1).

If $c=+\infty$ the theorem is proved. Suppose $c<\infty$. 

Let us first verify that $\liminf\limits_{t\rightarrow c-} x(t)>0$. 
By (a1) and continuity of the solution on $[0,c]$,
there exists $M>0$ such that $g(t,x(t)) \leq M$, $t\in [0,c]$. Following the above argument, we obtain
$\dot{x}(t)\geq -g(t,x(t))>-M$ and
$x(t)>x(0)e^{-Mc}$, which implies $\liminf_{t\rightarrow c-} x(t)>0$.
Thus
$\limsup\limits_{t\rightarrow c-} x(t)=+\infty$.

In fact, assuming the contrary that $\limsup_{t\rightarrow c-} x(t)<+\infty$ then, by (a3),
there exists $M_0>0$ such that $0 \leq x(t) \leq M_0$ on $[0,c)$ and $\varphi(t) \leq M_0$.
Since $f_k(\cdot, u_1,\dots,u_l)$ and $g(\cdot,u)$ are locally Lipschitz, they are locally essentially bounded for $t \in [0,c]$, 
$u,u_j \in [0,M_0]$,
thus $\dot{x}$ is also essentially bounded on $[0,c)$.

The solution satisfies
$$ x(t)= x(0)+ \int_0^t \dot{x}(s)~ds, ~~x(c)= x(0)+ \int_0^c \dot{x}(s)~ds, $$
thus the solution can be defined for $t \geq c$ and $[0,c)$ is not the maximum interval of existence. 
Thus, when justifying existence, we only need to prove boundedness of a solution on any finite interval.
 
Consider now the three cases.

1) Suppose $(a4_1)$ holds. Then there exists $t_0<c$ such that $h_j(t)\leq t_0,~ t\in [0,c)$, 
thus $|x(h_j(t))|\leq \max_{t\in [0,t_0]} |x(t)|<\infty$, hence
$$
0< x(t)\leq |x(0)|+\sum_{k=1}^m \int_0^c 
\sup_{0\leq s\leq t} f_k\left( s, x(h_1(s)),\dots,x(h_l(s)) \right) ds =A<\infty, ~~t \in [0,c),
$$
and therefore $\limsup_{t\rightarrow c-} x(t)=+\infty$ is impossible.

2) Suppose $(a4_2)$ holds. Then
$$
\dot{x}(t)\leq \sum_{k=1}^m \left[ \sum_{j=1}^l a_{kj}(t) x(h_j(t))+b_k(t) \right].
$$
If $\lim_{t \to c^-} x(t)=+\infty$ then there is $t_0 \in (0,c)$ such that 
\begin{equation}
\label{max1}
x(t_0)=\max_{s \leq t_0} x(s).
\end{equation}
Since $x(t)$ is positive on $(0,c)$,  on $[t_0,c)$ it does not exceed
the solution of the equation
$$\dot{z}(t)=\sum_{k=1}^m \left[ \sum_{j=1}^l a_{kj}(t) z(h_j(t))+b_k(t) \right], ~~z(t)=x(t), ~~t \leq t_0.$$
The function $z(t)$ is monotone nondecreasing on $[t_0,c)$, and from (\ref{max1}),
$$
z(t)=\max_{0 \leq s \leq t} z(s),~~z(t)\geq  z(h_j(t)),  ~~t \in [t_0,c), ~j=1, \dots, l.
$$ 
Thus $x(t) \leq y(t)$, 
where $y$ is a solution of the equation
$$\dot{y}(t) =a(t)y(t)+b(t),~~y(t_0)=x(t_0),~~a(t):=\sum_{k=1}^m \sum_{j=1}^l a_{kj}(t),
~b(t):=\sum_{k=1}^m b_k(t), ~~t\in [t_0,c) .$$
Hence
$$x(t) \leq y(t) \leq \int_{t_0}^c b(s) \exp \left\{ \int_{t_0}^c a(\tau)~d \tau \right\}~ds
+x(t_0)\exp\left\{ \int_{t_0}^c a(\tau)~d \tau \right\} =A< \infty, ~t \leq c,$$ 
since $a_{kj}$ and $b_k$ are integrable on $[t_0,c]$, and therefore there is a positive solution on $[0,\infty)$.

3) Suppose $(a4_3)$ holds. 
Since $(a4_3)$ is satisfied for $x$ large enough, we can find $A>0$ such that
the inequality in $(a4_3)$ holds for $x \geq u_j \geq A$, $j=1,\dots, m$. 
Let us choose $M \geq 2A$, $\displaystyle M \geq A+\sup_{t \leq 0} \varphi(t)$.
The function $g$ is locally Lipschitz, thus there is $\alpha>0$ such that 
$|g(t,x)-g(t,y)| \leq \alpha |x-y|$, $x,y \in [0,M]$, for any $t$.
We recall that there is $a_A$ such that $g(t,x)-\sum_{k=1}^m f_k(t, u_1,\dots,u_l)\geq a_A>0$ 
for $x\geq u_j \geq A$.

Denote $\varepsilon = \min\{ a_A/(2 \alpha), A \}$ then for $x,y \in [0,M]$, $|x-y|<\varepsilon$,
$$g(t,y) \geq g(t,x)-|g(t,x)-g(t,y)| \geq  g(t,x)-\alpha |x-y| \geq g(t,x)-\alpha \varepsilon \geq g(t,x)-\frac{a_A}{2}.$$ 
Therefore for $y \geq u_j-\varepsilon$, $j=1,\dots,l$, $x,y \in [0,M]$, $|x-y|\leq \varepsilon$
\begin{equation}
\label{auxil1}
g(t,y)-\sum_{k=1}^m f_k(t, u_1,\dots,u_l)\geq a_M/2>0.
\end{equation}
If $x(t) \geq M$ for some $t \in (0,c)$, denote 
$$t_0 =\inf \left\{ t \in [0,c] \left| x(t)=M \right. \right\}, ~~ 
t_1= \sup \left\{ t \in [0,t_0] \left| x(t)=M-\varepsilon \right. \right\}. $$
By definition $t_0>t_1>0$, and from continuity of $x$, $x(t_1)=M-\varepsilon$, $x(t_0)=M$, $x(h_j(t)) <M$, $t <t_0$, $j = 1, \dots, l$ and 
$x(t) \in (M-\varepsilon,M)$, $t \in (t_1,t_0)$. However, (\ref{auxil1}) implies
$$
\dot{x}(t)=\sum_{k=1}^m f_k(t, x(h_1(t)),\dots,x(h_l(t)))-g(t,x(t))\leq -a_M/2<0,~~t \in (t_1,t_0),
$$
thus  $x(t_1)> x(t_0)$, which contradicts to the assumption $x(t_1)=M-\varepsilon <x(t_0)=M$. 
Thus $x(t) \leq M$, $t\in [0,c]$; in fact, the inequality is satisfied for any $t$.
Hence a positive solution exists on $[0,\infty)$.
\end{proof}

\begin{remark}
The conditions of Theorem~\ref{theorem2} and \cite[Theorem 2.2]{JMAA2014}
are independent.
\end{remark}

\begin{example}\label{example1}
For the equation
\begin{equation}\label{7a}
\dot{x}(t)=x^2(t-\tau), ~~\tau>0
\end{equation}
condition $(a4_1)$ holds and $(a4_2), (a4_3)$ fail. It is interesting to note that
the equation $\dot{x}(t)=x^2$ with the initial condition $x(0)=x_0>0$ has the solution $x(t)=1/(x_0^{-1}-t)$
which only exists on $[0, 1/x_0)$.

For the equation
\begin{equation}\label{7b}
\dot{x}(t)=x(t-|\sin t|), 
\end{equation}
condition $(a4_2)$ holds and $(a4_1), (a4_3)$ fail.

For the equation
\begin{equation}\label{7c}
\dot{x}(t)=\frac{x^2(t-|\sin t|)}{1+x^2(t)}-x^3(t)
\end{equation}
condition $(a4_3)$ holds and $(a4_1), (a4_2)$ fail.

By Theorem \ref{theorem2}, problems for Eqs. (\ref{7a})-(\ref{7c}) with an initial function satisfying (a3), 
have a unique positive global solution. 
\end{example}

For the rest of the paper, we everywhere assume that problem (\ref{6}), (\ref{7}) 
has a unique positive global solution on $[0,\infty)$.

\section{Boundedness of Solutions}

Let us consider conditions under which all global solutions of (\ref{6}), (\ref{7}) are bounded.

\begin{guess}\label{theorem4}
Suppose conditions (a1)-(a3) hold.
Let also one of the following conditions be satisfied:

(a) $f_k(t,u_1,\dots,u_l)$ are strictly monotone increasing in $u_1,\dots, u_{l}$,
$$
\limsup_{u\rightarrow\infty}\frac{\sum_{k=1}^m f_k(t,u,\dots,u)}{g(t,u)}<1
$$
uniformly in $t$;

(b)   $f_k(t,u_1,\dots, u_l)$  are strictly monotone increasing in $u_j$ for some $j \in \{ 1,2, \dots, l \}$ and 
$$
\limsup_{u_j \rightarrow \infty}\frac{\sum_{k=1}^m f_k(t,u_1,\dots,u_l)}{g(t,u_j)}<1
$$
uniformly in $t$, $u_1,\dots, u_{j-1}, u_{j+1}\dots, u_l$.

Then any solution of  problem (\ref{6}),  (\ref{7}) is bounded.

If the following condition holds:

(c) $f_k(t,u_1,\dots, u_l)$  are strictly monotone increasing in $u_1,\dots,u_n$ for some $n \in \{ 1,2, \dots, l \}$ and there exists 
$M_0>0$
such that for any $M_1\geq M_0,\dots,M_{l-n}\geq M_0$
$$
\limsup_{u\rightarrow\infty}\frac{\sum_{k=1}^m f_k(t,u,\dots,u, M_1,\dots,M_{l-n})}{g(t,u)}<1
$$
uniformly in $t$, then there is no solution $x$ of  problem (\ref{6}),(\ref{7}) such that 
$\lim\limits_{t\rightarrow\infty} x(t)= \infty.$
\end{guess}
\begin{proof} 
Suppose that condition (a) holds and $x$ is an unbounded solution of problem (\ref{6}),  (\ref{7}). 
Let $A>\sup_{t\leq 0} \varphi(t)>0$ be a large number such that 
for some $\sigma >0$,  $\sum_{k=1}^m f_k(t,u,\dots,u) \leq (1-\sigma) g(t,u)$ for $u>A$. 
As $x$ is unbounded, for any fixed $M>A$ there exist points $t$ such that $x(t) \geq M$. 
Denote
$$t_1= \inf \{ t \geq 0 | x(t) \geq M \},$$
then $t_1>0$ as $M>A>\sup_{t\leq 0} \varphi(t)>0$ and $x(t)<M$ for $t \leq t_1$.
Let
$$
t_0 = \sup \{ t \leq t_1 | x(t) \leq A \}.$$
Since $x(t)<A$ for $t \leq 0$, we have $t_0>0$; by definition, $t_0<t_1$.
Also, $x(t) \geq A$ on $[t_0,t_1]$ with $A=x(t_0)<x(t_1)=M$.

Since $f_k$ are increasing in $u_j$ and $x(h_j(t)) <M$ on $(-\infty,t_1)$, we have
\begin{eqnarray*} 
& & \sum_{k=1}^m f_k(t, x(h_1(t)),\dots,x(h_l(t)))-g(t,x(t)) 
\\
& <  & \sum_{k=1}^m f_k(t, M,\dots,M)-g(t,x(t))
\\
& \leq & (1-\sigma) g(t,M)- g(t,x(t)) 
\\
& \leq &  g(t,M) - g(t,x(t)),~~
t \in [t_0,t_1]. 
\end{eqnarray*}
Thus the solution $x(t)$ of (\ref{6}),(\ref{7}) on $[t_0,t_1]$ does not exceed the solution of the initial value problem for the ordinary differential equation
\begin{equation}
\label{ordinary}
\dot{y}(t)= g(t,M)- g(t,y(t)), ~~y(t_0)=A<M,
\end{equation}
i.e. $x(t) \leq y(t)$, $t \in [t_0,t_1]$.
However, the solution of (\ref{ordinary}) satisfies $y(t)<M$, $t\geq t_0$. In fact, assuming the contrary, we obtain that $y(t^{\ast})=M$ for some $t^{\ast}>t_0$,
and there are two solutions through $(t^{\ast},M)$: $y$ and the one identically equal to $M$. This contradicts to the assumption of the local Lipschitz condition
which implies uniqueness. Thus $x(t_1) \leq y(t_1)<M$, and the contradiction with $x(t_1)=M$ proves boundedness of the solution $x$ of (\ref{6}),(\ref{7}).   

If condition (b) holds, the proof is similar to the previous case. Let for some $\sigma >0$,  $\sum_{k=1}^m f_k(t,u_1,\dots,u_l) \leq (1-\sigma) g(t,u_j)$ for $u_j>A$. Defining $A,M,t_0,t_1$ as previously, we obtain
\begin{eqnarray*}
&  & \sum_{k=1}^m f_k(t, x(h_1(t)),\dots,x(h_l(t)))-g(t,x(t)) \\
& < & \sum_{k=1}^m f_k(t, x(h_1(t)),\dots,x(h_{j-1}(t)),M,x(h_{j+1}(t)),
x(h_l(t)))-g(t,x(t)) \\
& \leq & (1-\sigma) g(t,M)- g(t,x(t)) \\
& \leq & g(t,M)- g(t,x(t)), ~~t \in [t_0,t_1].
\end{eqnarray*}
Again, comparing the solution $x(t)$ of (\ref{6}),(\ref{7}) on $[t_0,t_1]$ with the solution of (\ref{ordinary}) satisfying $y(t_1)<M$,
we obtain the contradiction $x(t) \leq y(t_1)<M$ with the assumption $x(t_1)=M$.

Finally, assume that condition (c) holds. 
Let $x$ be a solution of problem (\ref{6}),(\ref{7}) satisfying $\lim\limits_{t \to \infty} x(t)=\infty$. Since $\lim\limits_{t \to \infty} 
h_k(t)=\infty$, $k=1,\dots,l$, there exists $t_2 \geq 0$ such that $x(h_{n+1}(t))\geq M_0,\dots,x(h_l(t))\geq M_0$ for  $t\geq t_2$.

In addition, there is a number $A$, $A>\sup_{t\leq 0} \varphi(t)>0$, $A>\sup_{t\in [0,t_2]} x(t)>0$ such that 
$$\sum_{k=1}^m f_k(t,u,\dots,u, M_1,\dots,M_{l-n}) \leq (1-\sigma) g(t,u), ~~u \geq A,  ~~M_1\geq M_0,\dots,M_{l-n}\geq M_0.
$$
Fixing $M>A$ and choosing $t_1>t_0>t_2$ as previously such that $x(t)<M$ for $x<t_1$, $x(t_0)=A$ and $x(t) \in (A,M)$ for $t \in [t_0,t_1]$, 
we notice that
\begin{eqnarray*}
& & \sum_{k=1}^m f_k(t, x(h_1(t)),\dots,x(h_l(t)))-g(t,x(t)) \\ 
& < & \sum_{k=1}^m f_k(t, M,\dots,M, x(h_{n+1}(t)), \dots, x(h_{l}(t)))-g(t,x(t)) \\
& \leq & (1-\sigma) g(t,M)- g(t,x(t)) \\
& \leq & g(t,M)- g(t,x(t)), ~~t \in [t_0,t_1].
\end{eqnarray*}
Comparing the solution $x(t)$ of (\ref{6}),(\ref{7}) on $[t_0,t_1]$ with the solution of (\ref{ordinary}) satisfying $y(t_1)<M$,
we obtain a contradiction $x(t) \leq y(t_1)<M$ to the assumption $x(t_1)=M$.
Thus there are no solutions which tend to $+\infty$ as $t \to \infty$.
\end{proof}

\begin{remark}
The proof of Theorem~\ref{theorem4} implies that its conditions can be relaxed to
$$ \sum_{k=1}^m f_k(t,u,\dots,u) < g(t,u) $$
for any $t$ and $u$ large enough in (a),  
$$  
\sum_{k=1}^m f_k(t,u_1,\dots,u_l)< g(t,u_j)
$$
for any $t$, $u_1,\dots, u_{j-1}, u_{j+1}\dots, u_l$ as mentioned in (b)
and
$$
\sum_{k=1}^m f_k(t,u,\dots,u, M_1,\dots,M_{l-n}) < g(t,u)
$$
for any $t$ and for any $u$ large enough in (c).
\end{remark}

\begin{example}\label{example5}
Consider Eq.~(\ref{5}), 
where $a_k(t)\geq 0, b(t)\geq c(t)\geq 0, b(t)-c(t)\geq \beta>0$ are Lebesgue measurable bounded functions,   
for functions $h_k, g_k$ condition (a2) holds,  $n_k\geq 0, n\geq 0$.
Here condition $(a4_2)$ of Theorem~\ref{theorem2} holds, thus there exists a global positive solution of problem (\ref{5}), (\ref{7}).

Denote $f_k(t,u,v)=a_k(t)u/(1+v^{n_k})$, $g(t,u)=b(t)u-\frac{c(t)u}{1+u^n}$.
The functions  $f_k(t,u,v)$ are strictly monotone increasing in $u$. 
We have
$$
 \frac{\displaystyle \sum_{k=1}^m  f_k(t,u,v)}{g(t,u)}\leq \frac{\displaystyle \sum_{k=1}^m a_k(t)}{b(t)-c(t)}.
$$
Let
$$
\limsup_{t\rightarrow\infty} \frac{\sum_{k=1}^m a_k(t)}{b(t)-c(t)}<1,
$$
then there exists $t_0>0$ such that $\displaystyle \sup_{t \geq t_0} \frac{\sum_{k=1}^m a_k(t)}{b(t)-c(t)}<1$.
Shifting in Theorem~\ref{theorem4} (b) the initial point to $t_0$ and noticing that a continuous solution is bounded 
on $[0,t_0]$, we conclude that all solutions of Eq.~(\ref{5}) are bounded.

Evidently, condition (c) of Theorem~\ref{theorem4}  holds without any additional conditions.
Hence there is no solution satisfying $\lim\limits_{t\rightarrow\infty} x(t)= \infty$. 
\end{example}

\begin{example}\label{example5a}
Consider the equation
\begin{equation}\label{10}
\dot{x}(t)=a(t)x(t-h)x(t-g)-\left(b(t)-\frac{c(t)}{1+x^n(t)}\right)x^2(t),
\end{equation}
where $a(t)\geq 0$, $b(t)\geq c(t)\geq 0$, $b(t)-c(t)\geq \beta>0$, $a,b,c$ are Lebesgue measurable bounded functions, 
$h>0, g>0$, $n\geq 0$.
Here condition $(a4_1)$ of Theorem~\ref{theorem2} holds, thus there exists a global positive solution of problem (\ref{10}), 
(\ref{7}).

Denote $f(t,u,v)=a(t)uv$, $g(t,u)=\left(b(t)-\frac{c(t)}{1+u^n}\right)u^2$. The function $f$ is monotone increasing in both $u$
and $v$. 
We have
$$
\frac{f(t,u,u)}{g(t,u)}= \frac{a(t)}{b(t)-\frac{c(t)}{1+u^n}}\leq\frac{a(t)}{b(t)-c(t)}.
$$
Hence if 
$$
\limsup_{t\rightarrow\infty} \frac{a(t)}{b(t)-c(t)}<1
$$
then there exists $t_0>0$ such that $\displaystyle \sup_{t \geq t_0} \frac{\sum_{k=1}^m a_k(t)}{b(t)-c(t)}<1$.
Shifting in Theorem~\ref{theorem4} (a) the initial point to $t_0$ and noticing that the solution is bounded
on $[0,t_0]$, we conclude that all solutions of Eq.~(\ref{10}) are bounded.
\end{example}

We will give another statement on boundedness where monotonicity is not required.

\begin{guess}\label{theorem6}
Suppose conditions (a1)-(a3) hold, $g(t,u)\geq a_0(t) u$ for all $u \geq 0$ and 
$$
0\leq f_k(t,u_1,\dots,u_l)\leq 
 \sum_{j=1}^l A_{kj}(t)u_j+B_k \mbox{~~for all~~~} u_j \geq 0,$$ where $a_0(t)\geq 0, B_k\geq 0$,
$A_{kj}: [0,\infty) \to [0,\infty)$ are locally essentially bounded functions.

If the linear equation
\begin{equation}\label{9a}
\dot{x}(t)=-a_0(t)x(t)+ \sum_{k=1}^m \sum_{j=1}^l A_{kj}(t) x(h_j(t))
\end{equation}
is exponentially stable, then any positive solution of  problem (\ref{6}),  (\ref{7}) is bounded.
\end{guess}
\begin{proof}
If $x$ is a solution of (\ref{6}),  (\ref{7}) then
$$
 \dot{x}(t)\leq -a_0(t) x(t)+\sum_{k=1}^m \sum_{j=1}^l A_{kj}(t) x(h_j(t))+\sum_{k=1}^m B_k.
$$
Hence $x(t)\leq y(t)$ by  \cite[Corollary 2.2]{ABBD}, where  $y$ is a solution of the linear equation
$$
\dot{y}(t)= -a_0(t) y(t)+\sum_{k=1}^m \sum_{j=1}^l A_{kj}(t) y(h_j(t))+\sum_{k=1}^m B_k,~y(t)=x(t), ~~t\leq 0.
$$
Since  Eq.~(\ref{9a}) is exponentially stable,
$y$ is a bounded function. Hence $x$ is also a bounded function.
\end{proof}

\begin{example}\label{example5b}
Consider again Eq.~(\ref{5}) with the same conditions and notations as in Example~\ref{example5}.
We have $f_k(t,u,v_k)\leq a_k(t)u, g(t,u)\geq (b(t)-c(t))u$, $u\geq 0$. 
Hence if the linear equation
$$
\dot{x}(t)=-(b(t)-c(t))x(t)+\sum_{k=1}^l a_k(t)x(h_k(t))
$$
is exponentially stable, all solutions of Eq.~(\ref{5}) are bounded.
In particular, the condition  $\liminf\limits_{t\rightarrow\infty}(b(t)-c(t))>0$,$\limsup\limits_{t\rightarrow\infty}\frac{\sum_{k=1}^l a_k(t)} { b(t)-c(t)}<1$ 
 implies boundedness (see, for example, \cite[Corollary 1.4]{JMAA1996}).
\end{example}
 
\begin{corol}\label{corollary2}
Suppose conditions (a1)-(a3) hold, $g(t,u)\geq a_0 u>0$ for $u>0$, and
$$
\limsup_{u_j\rightarrow\infty, j=1,\dots,l}f_k(t,u_1,\dots,u_l)\leq B_k.
$$
Then any solution of  problem (\ref{6}),  (\ref{7}) is bounded.
\end{corol}
 
\begin{example}\label{example2}
Consider the Mackey-Glass type equation
\begin{equation}\label{9}
\dot{x}(t)=\sum_{k=1}^m  \frac{a_k(t)|\sin (x(h_k(t)))|}{1+x^{n_1}(h_1(t))+\dots+x^{n_l}(h_l(t))}-\left(b(t)
+\frac{c(t)}{1+x^n(t)}\right)x(t),~ t \geq 0,
\end{equation}
where $ a_k, b, c$ are nonnegative essentially bounded on $[0,\infty)$ functions, $b(t)+c(t)\geq \beta> 0, n_j\geq 0$, $n>0$.
Here condition $(a4_3)$ of Theorem~\ref{theorem2} holds, thus there exists a global positive solution of problem (\ref{9}), (\ref{7}).
Denote
$$
 f_k(t,u, u_1,\dots,u_l)=\frac{a_k(t)|\sin u|}{1+u_1^{n_1}+\dots+u_l^{n_l}}, ~~ g(t,u)=b(t)u+\frac{c(t)u}{1+u^n}.
$$
Hence 
$g(t,u)\geq \beta u$, and the
functions $f_k(t,u, u_1,\dots,u_l)$ are bounded.  By Corollary \ref{corollary2}, all solutions of 
Eq.~(\ref{9}) are bounded.

Let us note that Theorem \ref{theorem4} cannot be applied to (\ref{9}), as the functions $f_k(t,u, u_1,\dots,u_l)$ are not 
monotone increasing in $u$.
\end{example}

\section{Persistence of Solutions}

We proceed now to persistence and permanence of solutions. As previously,
we everywhere assume that problem (\ref{6}), (\ref{7}) has a unique positive global solution on $[0,\infty)$.

\begin{definition}
A positive solution $x(t)$ is {\bf persistent} if ~$\displaystyle \liminf\limits_{t \to \infty} x(t) >0$ and 
is {\bf permanent} if it is also bounded.
\end{definition}

\begin{guess}\label{theorem7}
Suppose that conditions (a1)-(a3) are satisfied.

(a) If $f_k(t,u_1,\dots,u_l)$ are strictly monotone increasing in $u_1,\dots, u_l$ and
$$
\liminf_{u\rightarrow 0^+}\frac{\sum_{k=1}^m f_k(t,u,\dots,u)}{g(t,u)}>1
$$
uniformly on $t\in [0,\infty)$ then any solution $x$ of (\ref{6}),  (\ref{7}) satisfies 
$\liminf\limits_{t\rightarrow\infty} x(s)>0$.

(b)  If $f_k(t,u_1,\dots,u_l)$ are strictly monotone increasing in $u_1, \dots, u_n$ for some $n\in \{1,\dots,l\}$,
 monotone decreasing in $u_{n+1}, \dots, u_l$, and there exists $M_0>0$ such that for any $0<M_1,\dots,M_{l-n}\leq M_0$ we have
$$
\liminf_{u\rightarrow 0^+}\frac{\sum_{k=1}^m f_k(t,u,\dots,u,M_1,\dots,M_{l-n})}{g(t,u)}>1 
$$
uniformly on $t\in [0,\infty)$, then there is no solution $x$ of (\ref{6}),  (\ref{7}) satisfying
$\lim\limits_{t\rightarrow\infty} x(s)=0$.
\end{guess}
\begin{proof} 
First assume that the assumption in (a) holds.
Suppose that $x$ is a solution of (\ref{6}),  (\ref{7}) such that $\liminf\limits_{t\rightarrow\infty} x(t)=0$.

The solution is positive, we can consider $t_0$ such that $h_k(t)>0$ for $t>t_0$, and reduce ourselves
to $t>t_0$. 
Then there exist $\sigma >0$ and $b>0$ small enough such that
$\sum_{k=1}^m f_k(t,u,\dots,u) \geq (1+\sigma) g(t,u)$ for $u \in (0,b)$.

Let us note that, as the solution is positive, $\displaystyle \min_{t\in [0, t_0]} x(t)>0$ 
and thus we can choose $m<a<b$ such that also $\displaystyle 0<m<a<\min_{t \in [0, t_0]} x(t)$.

As $\liminf\limits_{t\rightarrow\infty} x(t)=0$, there exist points $t$ such that $x(t) \leq m$. 
Denote
$$t_2= \inf \{ t \geq t_0 | x(t) \leq m \},$$
then $t_2>t_0$ 
and $x(t)>m$ for $t \in [0,t_2)$.
Let
$$
t_1 = \sup \{ t \leq t_2 | x(t) \geq a \}.$$
The inequality $x(t)>a$ for $t \in [0,t_0]$ implies $t_1>t_0$; by definition, $t_1<t_2$.
We have $m \leq x(t) \leq a$ on $[t_1,t_2]$ with $a=x(t_1)>x(t_2)=m$.

Since $f_k$ are increasing in $u_j$ and $x(h_j(t))>m$ on $(t_0,t_2)$,
\begin{eqnarray*}
& & \sum_{k=1}^m f_k(t, x(h_1(t)),\dots,x(h_l(t)))-g(t,x(t)) 
\\
& >  & \sum_{k=1}^m f_k(t, m,\dots,m)-g(t,x(t)) \\
& \geq & (1+\sigma) g(t,m)- g(t,x(t)) \\ &  \geq & g(t,m)- g(t,x(t)),~~
t \in [t_1,t_2]. 
\end{eqnarray*}
Thus the solution $x(t)$ of (\ref{6}),(\ref{7}) on $[t_1,t_2]$ is not less than the solution of the initial value problem for the ordinary differential equation
\begin{equation}
\label{ordinary1}
\dot{y}(t)= g(t,m)- g(t,y(t)), ~~y(t_1)=a>m,
\end{equation}
i.e. $x(t) \geq y(t)$, $t \in [t_1,t_2]$.
However, the solution of (\ref{ordinary1}) satisfies $y(t)>m$, $t\geq t_1$. In fact, assuming the contrary, we obtain that $y(t^{\ast})=m$ for some $t^{\ast}>t_1$,
and there are two solutions through $(t^{\ast},m)$: $y$ and the one identically equal to $m$. This is impossible as $g$
is locally Lipschitz which implies uniqueness. Thus $x(t_2) \geq y(t_2)>m$ which contradicts to the assumption $x(t_2)=m$. Hence all solutions are persistent.

Next, let us assume that the conditions in (b) hold and $x(t) \to 0$ as $t \to \infty$.
Let $\bar{t}$ be such that $x(t) \leq M_0$ for $t \geq \bar{t}$ and 
$t_0 \geq \bar{t}$ such that $h_j(t) \geq \bar{t}$ for $t \geq t_0$, $j=1, \dots, l$.

Thus $x(h_j(t))\leq M_0$ for $t \geq t_0$. Next, there
are $\sigma >0$ and $a>0$ small enough such that $\sum_{k=1}^m f_k(t,u,\dots,u,M_1,\dots,M_{l-n}) \geq (1+\sigma) g(t,u)$ for $u \in (0,a)$
and any $0<M_1,\dots,M_{l-n}\leq M_0$. Let $0<m<a$; as $x(t) \to 0$, there is a $t$ such that $x(t) \leq m$. Introducing
$$
t_2= \inf \{ t \geq t_0 | x(t) \leq m \}, ~~ t_1 = \sup \{ t \leq t_2 | x(t) \geq a \},
$$
we notice that $x(h_j(t))<M_0$, $m<x(t)<a$ for $t \in [t_1,t_2]$, $x(t_1)=a>x(t_2)=m$. Therefore
\begin{eqnarray*}
& & \sum_{k=1}^m f_k(t, x(h_1(t)),\dots,x(h_l(t)))-g(t,x(t)) \\ & > & \sum_{k=1}^m f_k(t, m,\dots,m,x_{n+1}(h_{n+1}(t)), \dots, x(h_l(t)) )-g(t,x(t))
\\ & \geq &  (1+\sigma) g(t,m)- g(t,x(t)) \\ & \geq &  g(t,m)- g(t,x(t)) ,~~
t \in [t_1,t_2]. 
\end{eqnarray*}
Thus the solution $x(t)$ of (\ref{6}),(\ref{7}) on $[t_1,t_1]$ is not less than the solution of the initial value problem (\ref{ordinary1}) which, as in case (a), satisfies $y(t_2)>m$. Hence $x(t_2) \geq y(t_2)>m$, the contradiction with $x(t_2)=m$ yields that the solution does not tend to zero.
\end{proof}

\begin{remark}
The proof of Theorem~\ref{theorem7} implies that its conditions in fact can be relaxed to
$$
\sum_{k=1}^m f_k(t,u,\dots,u) > g(t,u)
$$
for any $t\in [0,\infty)$ and $u>0$ small enough in (a) and
$$
\sum_{k=1}^m f_k(t,u,\dots,u,M_1,\dots,M_{l-n}) > g(t,u)
$$
for any $t\in [0,\infty)$, $u>0$ small enough and $0<M_1,\dots,M_{l-n}\leq M_0$ in (b).
\end{remark}

\begin{example}\label{example7}
Consider Eq.~(\ref{5}) with the same conditions as in Example~\ref{example5}.
We also use the same notations as in  Example~\ref{example5}.
For Eq.~(\ref{5}), we have 
$$
\liminf_{u\rightarrow 0^+}\frac{\sum_{k=1}^m f_k(t,u,v)}{g(t,u)}\geq \sum_{k=1}^m\frac{a_k(t)}{b(t)(1+v^{n_k})}
$$
There exist $M_0>0$ and $t_0 \geq 0$ such that the condition $\displaystyle \liminf_{t\rightarrow \infty} 
\sum_{k=1}^m\frac{a_k(t)}{b(t)}>1$ implies  $$ 
\sum_{k=1}^m \frac{a_k(t)}{b(t)(1+M^{n_k})}>1$$ for $M\leq 
M_0$ and $t \geq t_0$.
In Theorem~\ref{theorem7} (b) we shift the initial point to $t_0$ and notice that the bounds of the positive solution on 
$[0,t_0]$ do not influence the asymptotics.
Hence for $\displaystyle \liminf_{t\rightarrow \infty} \sum_{k=1}^m\frac{a_k(t)}{b(t)}>1$, 
there is no solution $x$ of Eq.~(\ref{5}) satisfying $\lim\limits_{t\rightarrow\infty} x(s)=0$.
\end{example}

Further we illustrate in Example~\ref{MG_is_bad} that the conditions in Theorem~\ref{theorem7} (b) are not sufficient to establish permanence of solutions.

\begin{example}\label{example7a}
Consider Eq.~(\ref{10}) with the same conditions as in Example~\ref{example5a}.
We also use the same notations as in  Example~\ref{example5a}.
For Eq.~(\ref{10}), we have 
$$
\liminf_{u\rightarrow 0+}\frac{f(t,u,u)}{g(t,u)}\geq \frac{a(t)}{b(t)}\, .
$$
Let $\liminf\limits_{t\rightarrow \infty} \frac{a(t)}{b(t)}>1$, then $\inf_{t \geq t_0} \frac{a(t)}{b(t)}>1$
for some $t_0 \geq 0$. Shifting the initial point to $t_0$ in Theorem~\ref{theorem7} (a) and noticing
that the solution is positive on $[0,t_0]$, we conclude that $\liminf\limits_{t\rightarrow \infty} \frac{a(t)}{b(t)}>1$
implies that all solutions of Eq.~(\ref{10}) are
persistent.

Therefore if
$$
\liminf_{t\rightarrow \infty} \frac{a(t)}{b(t)}>1,~\limsup_{t\rightarrow \infty} \frac{a(t)}{b(t)-c(t)}<1
$$
then Eq.~(\ref{10}) is permanent.
\end{example}

Everywhere above, we only assumed that (a2) is satisfied, i.e. the arguments of $x$ tend to $\infty$ as $t \to \infty$.
In the following theorem we assume a stronger condition that the delays are bounded.

\begin{guess}\label{theorem8a} 
Suppose conditions (a1)-(a3) are satisfied, $f_k(t,u_1,\dots,u_l)$ are monotone increasing in $u_1, \dots, u_n$ for some $n\in 
\{1,\dots,l\}$, monotone decreasing in $u_{n+1}, \dots, u_l$, and there exist
constants $\tau>0$, $A>0$, $\mu>0$, $M>0$, $0<\beta<B$ such that $t-\tau<h_j(t)\leq t$, $j=1, \dots, 
l$, $\displaystyle \sum_{i=1}^m f_i(t,u_1,\dots,u_l)\leq A u_j$, $u_j>0$, for some $j \in \{1, \dots, n\}$, $0<\beta u \leq g(t,u) \leq Bu$ for 
$u>0$. If there exists $M>0$ such that
\begin{equation}
\label{8a2}
\limsup_{t \to \infty} \frac{\sum_{i=1}^m f_i(t,u,\dots,u, M, \dots, M)}{g(t,u)} <1 
\end{equation}
uniformly on $u \in [M,\infty)$ then any solution $x$ of (\ref{6}),(\ref{7}) is bounded, with the upper bound
\begin{equation}
\label{estimate_upper}
\limsup_{t \to \infty} x(t) \leq M e^{2(A+B)\tau}.
\end{equation} 

If there exists $\mu >0$ such that
\begin{equation}
\label{8a1}
\liminf_{t \to \infty} \frac{ \sum_{i=1}^m f_i(t,u, \dots, u, \mu,\dots,\mu)}{g(t,u)} > 1 
\end{equation}
uniformly on $u \in [0,\mu]$ then any solution $x$ of (\ref{6}),  (\ref{7}) is persistent, and
\begin{equation}
\label{estimate_lower}
\liminf_{t \to \infty} x(t) \geq \mu e^{-2B\tau}.
\end{equation}  
\end{guess}
\begin{proof} 
Let $x$ be a solution of (\ref{6}),(\ref{7}). First we will prove that $x$ is bounded 
and obtain an eventual upper estimate for $x$, then we justify permanence and present an eventual lower estimate.
As a preliminary work, possible growth and decrease of $x$ is estimated.

Let us consider $t_{\ast}$ large enough such that for some $\tilde{t}$,
$$\frac{\sum_{i=1}^m f_i(t,u,\dots,u, M, \dots, M)}{g(t,u)} \leq \alpha<1,~~t \geq \tilde{t}, ~~u \geq M,$$
and $h_k(t) \geq \tilde{t}$ for $t \geq t_{\ast}$ (we can take $t_{\ast}=\tilde{t}+\tau$). 
Denote $t_1=t_{\ast}+\tau$, $t_j = t_{\ast}+j \tau$.
The solution is positive and continuous, so it is possible to introduce a series of maximum and minimum
values on $[t_{j-1},t_j]$: 
$$m_j = \min_{t\in [t_{j-1},t_j]} x(t), ~~M_j= \max_{t\in [t_{j-1},t_j]} x(t), ~~j \in {\mathbb N}.$$ 
Let $x(t_{j-1}^{\ast})=M_{j-1}$, where $t_{j-1}^{\ast}\in [t_{j-2},t_{j-1}]$.
Also, $\dot{x}(t) \geq -g(t,x(t)) \geq - B x(t)$, thus 
$$ x(t) \geq x(t_{j-1}^{\ast})e^{-B(t-t_{j-1}^{\ast})}= M_{j-1} e^{-B(t-t_{j-1}^{\ast})} \geq M_{j-1} e^{-2B\tau},
~~t \in [t_{j-1},t_j],$$
since $t-t_{j-1}^{\ast} \leq t_j-t_{j-2} = 2\tau$. Thus, $m_j \geq M_{j-1} e^{-2B\tau}$. 

Next, let us develop an upper estimate. By the assumptions of the theorem, the solution satisfies $x(t_{j-1})\leq M_{j-1}$ and
\begin{eqnarray*}
\dot{x}(t) & = & \sum_{k=1}^m f_k(t, x(h_1(t)),\dots,x(h_l(t)))-g(t,x(t)) < \sum_{k=1}^m f_k(t, x(h_1(t)),\dots,x(h_l(t)))
\\
& \leq & A x(h_j(t)) \leq A \max\left\{ M_{j-1}, \max_{s \in [t_{j-1},t]} x(s) \right\}, ~~ t \in [t_{j-1},t_j].
\end{eqnarray*}
Hence $x(t)$ is less than the solution of the initial value problem $\dot{x}(t)=A x(t)$, $x(t_{j-1})=M_{j-1}$,
which is $M_{j-1}\exp(A(t-t_{j-1}))$, therefore
\begin{equation}
\label{upper_1a}
M_j \leq M_{j-1} e^{A\tau}, ~~x(t) \leq M_{j-1} e^{2A\tau}, ~~ t \in [t_j,t_{j+1}].
\end{equation}

For the sake of contradiction, let us assume that the solution $x$ is unbounded, i.e. for any $\overline{M} > Me^{2(A+B)\tau}$,
where $M$ is described in the conditions of the theorem,
there is an interval $[t_j,t_{j+1}]$ where the inequality $x(t)\geq \overline{M}$ is attained for the first time.
Hence there exists $t^{\ast}$ where $x(t^{\ast})=\overline{M}$ and $\varepsilon>0$ such that $t \in [t^{\ast}-\varepsilon,t^{\ast}] \subset
[t_j,t_{j+1}]$ and $x(t)= \sup_{s \in [0,t]} x(s)$, $t \in [t^{\ast}-\varepsilon,t^{\ast}]$.

According to estimate (\ref{upper_1a}), $\overline{M} > Me^{2(A+B)\tau}$ implies $M_{j-1} \geq Me^{2B\tau}$, 
while $m_j \geq M_{j-1} e^{-2B\tau}$ yields that $m_j \geq M$ and also $x(t) \geq M$ on $[t_{j-1},t^{\ast}]$. Thus, all $x(h_i(t))\geq M$ for $t \in [t_j,t_{j+1}]$,
$i=1, \dots, l$, and
for $t \in [t^{\ast}-\varepsilon,t^{\ast}]$,
\begin{eqnarray*}
\dot{x}(t) & \displaystyle = & \sum_{i=1}^m f_i(t, x(h_1(t)),\dots,x(h_l(t)))-g(t,x(t)) \\ & 
\displaystyle \leq & \sum_{i=1}^m f_k(t, x(t), \dots, x(t), M, \dots, M)- g(t,x(t)) \\ & 
\leq & \alpha g(t,x(t))- g(t,x(t))= -(1-\alpha) g(t,x(t)) < 0,
\end{eqnarray*}
which contradicts to the assumption $x(t^{\ast}-\varepsilon) \leq x(t^{\ast})=\overline{M}$. Thus, the solution is bounded
with the eventual upper bound of $Me^{2(B+A)\tau}$. 

Next, let us proceed to persistence and assume that for $t \geq t_{\ast}-\tau $,
$$\frac{\sum_{i=1}^m f_i(t,u,\dots,u, \mu, \dots, \mu)}{g(t,u)} \geq C > 1,~~t \geq \tilde{t}, ~~0 \leq u \leq \mu,$$
and  introduce $t_j$, $m_j$ and $M_j$ as previously.
If $\liminf\limits_{t\to\infty} x(t)=0$ then there exist $t^{\ast}$ large enough
and $\varepsilon$ small enough such that 
$x(t)= \min_{s \in [0,t]} x(s)$ and $x(t)< \mu e^{-2B\tau}$ on $[t^{\ast}-\varepsilon,t^{\ast}] \subset 
[t_j,t_{j+1}]$. As previously, we obtain $x(t) < \mu$ on $[t_{j-1},t^{\ast}]$, so
$x(h_i(t)) < \mu$, $i=1,\dots,l$, $t \in [t^{\ast}-\varepsilon,t^{\ast}]$. On 
$[t^{\ast}-\varepsilon,t^{\ast}]$, we have $x(h_j(t)) < \mu$, $x(t)<\mu$ and 
\begin{eqnarray*}
\dot{x}(t) & \displaystyle = & \sum_{i=1}^m f_i(t, x(h_1(t)),\dots,x(h_l(t)))-g(t,x(t)) \\ & 
\geq & \displaystyle \sum_{i=1}^m f_k(t, x(t), \dots, x(t), \mu, \dots, \mu)- g(t,x(t)) \\ &
\geq & Cg(t,x(t))-g(t,x(t))=(C-1)g(t,x(t)) >0,
\end{eqnarray*}
which contradicts to the assumption $x(t^{\ast}-\varepsilon) \geq x(t^{\ast})$. Thus, the solution is also persistent and satisfies (\ref{estimate_lower}).
\end{proof}

\begin{example}
\label{example_MG}
Consider the Mackey-Glass equation
\begin{equation}
\label{ex2eq1}
\dot{x}(t)= \frac{a(t)x(h(t))}{1+x^{n}(p(t))}-b(t)x(t),
\end{equation}
where $a$ and $b$ are Lebesgue measurable bounded functions satisfying
$0\leq \alpha \leq a(t) \leq A$, $0<\beta \leq b(t) \leq B$, $t-h(t)\leq \tau$, $t-p(t) \leq \tau$, $n >0$.
The bounds for $a(t)$ and $b(t)$ guarantee that $\limsup\limits_{t\rightarrow\infty}\frac{a(t)}{b(t)}$ is finite.
Thus inequality \eqref{8a2} is satisfied for any $M>M_0$, where
$$
M_0 =\left\{\begin{array}{ll}
1,& \limsup\limits_{t\rightarrow\infty}\frac{a(t)}{b(t)}\leq 1,\\
\limsup\limits_{t\rightarrow\infty}\left(\frac{a(t)}{b(t)}-1\right)^{\frac{1}{n}},&
\limsup\limits_{t\rightarrow\infty}\frac{a(t)}{b(t)}>1.
\end{array}\right.
$$
Thus, all solutions of \eqref{ex2eq1} are bounded, with the eventual upper bound of
$\displaystyle
M e^{2(A+B)\tau}
$.
Assume now that in addition
\begin{equation}\label{per}
\liminf_{t\rightarrow\infty}\frac{a(t)}{b(t)} > 1.
\end{equation}
Inequality \eqref{8a1} is valid for any $0<\mu<\mu_0$, where
$$
\mu_0=\liminf_{t\rightarrow\infty}\left(\frac{a(t)}{b(t)}-1\right)^{\frac{1}{n}}.
$$
Hence, if condition (\ref{per}) holds, then  any positive solution $x$ is persistent with
$$
\liminf_{t \to \infty} x(t) \geq  \mu_0 e^{-2B\tau}.
$$
Moreover, condition  (\ref{per}) implies permanence of all  positive solutions of Eq.~(\ref{ex2eq1}).
\end{example}

The following example illustrates the fact that boundedness of delays in Theorem~\ref{theorem8a}  is required to conclude that  
all solutions of Eq.~(\ref{6}) are bounded and persistent.

\begin{example}
\label{MG_is_bad}
Consider the equation
\begin{equation}
\label{eq_ex1}
\dot{x}(t)=\frac{a(t)x(h(t))}{1+x^2(g(t))}-x(t)
\end{equation}
with piecewise constant $h(t)$ and $g(t)$.
Let us note that the equation $\dot{x}+x(t)=A$, $x(t_0)=x_0$ has the solution
$x(t)=(x_0-A)\exp\{-(t-t_0)\} +A$, so for any $B$ between $A$ and $x_0$ there is
a finite $t_1>t_0$ such that $x(t_1)=B$, $t_1=t_0+ \ln((x_0-A)/(B-A))$.

Let $t_1<t_2< \dots $ be a sequence of positive numbers such that
$$
a(t) = \left\{ \begin{array}{ll} 2, & t\in [t_{2k},t_{2k+1}), \\
6, & t\in [t_{2k+1},t_{2k+2}),  \end{array} \right. 
$$
$$
h(t)= \left\{ \begin{array}{ll} t_{2k-1}, & t\in [t_{2k},t_{2k+1}), \\  
t_{2k}, & t\in [t_{2k+1},t_{2k+2}),  \end{array} \right. \quad
g(t)= \left\{ \begin{array}{ll} t_{2k}, & t\in [t_{2k},t_{2k+1}), \\
t_{2k-1}, & t\in [t_{2k+1},t_{2k+2}),  \end{array} \right.
$$
where $t_0=0$, $x(t_0)=1$, $t_{-1}=-1$, $x(t_{-1})=\varphi(-1)=\frac{1}{4}$.
We justify that we can find $t_i$ such that
$$x(t_{2k})=2^k, ~~x(t_{2k+1})=2^{-k-1}, ~~k \in {\mathbb N}.$$
In fact, on $[0,t_1]$ we have $x(h(t))=\frac{1}{4}$, $x(g(t))=1$, $a(t)=2$, the initial value problem is 
$\dot{x}(t)+x(t)=\frac{1}{4}$, $x(0)=1$,
so we can find $t_1$ such that $x(t_1)=\frac{1}{2}$. 

On $[t_1,t_2]$, $x(h(t))=1$, $x(g(t))=\frac{1}{4}$,
$a(t)=6$, the initial value problem is
$\dot{x}(t)+x(t)=6/(1+1/16)>2$, $x(t_1)=1/2$, so there is $t_2$ such that $x(t_2)=2$.

Let us proceed to the induction step. If $x(t_{2k})=2^k$, $x(t_{2k+1})=2^{-k-1}$
then on $[t_{2k},t_{2k+1}]$ we have the initial value problem
$$\dot{x}(t)+x(t)=\frac{2 \cdot 2^{-k-1}}{1+2^{2k}}<2^{-k-2},\quad k \in {\mathbb N}, \quad x(t_{2k})= 2^k,$$
thus there exists $t_{2k+1}$ such that $x(t_{2k+1})=2^{-k-2}$. On 
$[t_{2k+1},t_{2k+2}]$, we have the initial value problem
$$\dot{x}(t)+x(t)=\frac{6 \cdot 2^k}{1+2^{-2k-2}}>2^{k+1}, \quad k \in {\mathbb N}, \quad x(t_{2k+1})= 2^{-k-2},
$$
hence there is $t_{2k+2}$ such that $x(t_{2k+2})=2^{k+1}$, which concludes the induction step.
Here both $b$ and $a$ are bounded, separated from zero, $a/b \geq 2>1$, $g$ and $h$ satisfy (a2)
but the solution is neither bounded nor persistent. In this example, the delays $h$ and $g$ are unbounded.
\end{example}

\section{Unbounded Solutions}

Let us consider the case when positive solutions are unbounded.

\begin{guess}\label{theorem9}
Suppose $f_k(t,u_1, \dots, u_l)$, $k=1, \dots, m$ are increasing functions 
in $u_1,\dots,u_l$ for any $t$, 
there is $K_0>0$ such that for any $K\geq K_0$ there exists $a_K>0$ such that
$$
\inf_{t \geq 0} \left[ \sum_{k=1}^m f_k (t,K, \dots, K)-g(t,K) \right] \geq a_K.
$$
For any  $K\geq K_0$, if $\varphi(t)>K$ for $t\leq 0$ then 
the solution of (\ref{6}),  (\ref{7}) satisfies
$$\limsup_{t\rightarrow\infty}x(t)=+\infty.$$
\end{guess} 
\begin{proof}  
Suppose that $x$ is a solution of Eq.~(\ref{6}) such that $x(t)>K \geq K_0$, $t\leq 0$.

First, let us prove that $x(t)>K_0$ for any $t\geq 0$. Assume that it is not so and there exists
$t_0>0$ such that $x(t)>K \geq K_0$, $t\in [0,t_0)$, $x(t_0)=K_0$. 
In some left neighbourhood $[t_0-\varepsilon,t_0)$ of $t_0$ we have $K_0 < x(t)$ and $|g(t,x(t))-g(t,K_0)|<a_{K_0}/2$.
Hence
\begin{eqnarray*}
\dot{x}(t) & = & \sum_{k=1}^m f_k (t,x(h_1(t)), \dots, x(h_l(t)))-g(t,x(t))
\\
& \geq & \sum_{k=1}^m f_k(t,K_0, \dots, K_0)-g(t,K_0)+g(t,K_0)-x(t,x(t))
\\
&\geq & a_{K_0}-a_{K_0}/2=a_{K_0}/2>0,~~t \in (t_0-\varepsilon,t_0),
\end{eqnarray*}
which implies
$$
K_0= x(t_0) = x(t_0-\varepsilon)+ \int_{t_0-\varepsilon}^{t_0} \dot{x}(s)\, ds \geq x(t_0-\varepsilon)+(a_K/2) \varepsilon > K_0.
$$
The contradiction proves $x(t) > K_0$ for any $t\geq 0$.
  
Next, let us define
$$
K_1 = \sup \left\{ u> K_0 \left|  \inf_{t \geq 0, x \in [K_0,u]} \left[\sum_{k=1}^m f_k (t,K_0, \dots, K_0)-g(t,x) \right] 
\geq \frac{1}{2} a_{K_0} \right. \right\},
$$  
$$
K_1^{\ast} = \sup \left\{ u > K_0 \left|  \inf_{t \geq 0,x \in [K_0,u] } 
\left[\sum_{k=1}^m f_k (t,K_0, \dots, K_0)-g(t,x) \right] \geq 0  \right. \right\}.
$$
Either $K_1=+\infty$ or  $K_1^{\ast}=+\infty$ would imply that
$$
\dot{x}(t) \geq \sum_{k=1}^m f_k (t,K_0, \dots, K_0)-g(t,x) >0
$$
and thus $x(t)$ is increasing for any $t$; if $K_1^{\ast}=+\infty$ then
it is increasing with the guaranteed rate $\dot{x}(t) \geq  \frac{1}{2} a_{K_0}$, and the solution is obviously unbounded. 
 
By (a1), $g$ is locally Lipschitz, hence there exists $\alpha = \alpha([K_0,2K_0])$  such that 
$$\inf_{t \geq 0} |g(t,u)-g(t,y)| \leq  \alpha 
|u-y|, ~~u,y \in [K_0,2K_0].$$ Denote $\displaystyle \sigma :=\min\left\{  \frac{a_{K_0}}{2 \alpha}, K_0 \right\}$.
Thus, for $\displaystyle x \in [K_0,K_0+\sigma]$ we have
\begin{eqnarray*}
& & \inf_{t \geq 0} \left[\sum_{k=1}^m f_k (t,K_0, \dots, K_0)-g(t,x) \right] \\ 
& \geq & \inf_{t \geq 0} \left[\sum_{k=1}^m f_k(t,K_0, \dots, K_0)-g(t,K_0) \right] -  |g(t,K_0)-g(t,x)| \\
& \geq & a_{K_0} -  \alpha|x-K_0| \geq a_{K_0} - \alpha  \frac{a_{K_0}}{2 \alpha} = a_{K_0} - \frac{1}{2}a_{K_0}=\frac{a_{K_0}}{2}.
\end{eqnarray*}
Therefore $K_1 \geq K_0+\sigma$. Similarly, $K_1^{\ast} \geq K_1+\sigma$.

As long as $x(t) \in (K_0,K_1]$, we have $\dot{x}(t) \geq \frac{1}{2} a_{K_0}$, thus $x(t)>K_1$ for some $t$; moreover, $x(t)>K_1$ for $t$ large enough. 
In fact, assuming that there is an interval $[t_1-\varepsilon,t_1)$ where $x(t) \in (K_1,K_1^{\ast})$
while $x(t_1)=K_1$, we notice that, due to the fact that $x(t)>K_0$ for $t \geq 0$, 
\begin{eqnarray*}
\dot{x}(t) & = & \sum_{k=1}^m f_k (t,x(h_1(t)), \dots, x(h_l(t)))-g(t,x(t)) 
\\
& \geq & f(t,K_0, \dots, K_0)-g(t,x(t)) \geq 0
\end{eqnarray*}
for $t \in [t_1-\varepsilon,t_1]$,
which excludes the possibility $x(t_1-\varepsilon) > x(t_1)$. Thus, $x(t)>K_1$ for $x$ large enough, say, for $t>t_1^{\ast}$.


Further, we consider $t$ large enough such that $h_i(t)>t_1^{\ast}$ for any $i=1, \dots, l$. Denote
$$
K_2 = \sup \left\{ u > K_1 \left|  \inf_{t \geq 0, x \in [K_1,u]} \left[\sum_{k=1}^m f_k (t,K_1, \dots, K_1)-g(t,x(t)) \right] \geq \frac{1}{2} 
a_{K_{1}} \right.   \right\}.
$$  
Similarly to the previous argument we verify that $x(t)>K_2$ for $t$ large enough, denote $K_n$, 
$n \in {\mathbb N}$ and repeat this procedure. 
Thus there is an increasing sequence of positive numbers $K_1<K_2< \dots < K_n < \dots$, if finite, 
and points $t_1 \leq t_2 \leq \dots
\leq t_n \leq \dots$ such that $x(t) \geq K_n$ for $x \geq t_n$ and
$$
\inf_{t \geq 0} \left[\sum_{k=1}^m f_k (t,K_n, \dots, K_n)-g(t,K_{n-1}) \right] = \frac{1}{2} a_{K_{n-1}}.
$$
If at least one of $K_n$ is infinite, the solution tends to infinity, as explained earlier.
In addition, for $\lim_{n \to \infty} K_n = +\infty$, the solution is unbounded. Assuming that $\lim\limits_{n \to \infty} K_n = d < 
+\infty$ and proceeding to the limit in $n$ in the above inequality, we obtain
$$
\inf_{t \geq 0} \left[\sum_{k=1}^m f_k (t,d, \dots, d)-g(t,d) \right] = \frac{1}{2} a_d,
$$
which contradicts to the assumption of the theorem that this infimum is not less than $a_d$.
\end{proof}

\begin{example}
\label{ex_unbounded}
Consider Eq.~(\ref{10}) with the same conditions as in Example~\ref{example5a} and in addition
\begin{equation}
\label{ex_unb1}
\inf_{t \geq 0} \left[ a(t)-b(t) \right] \geq \alpha > 0.
\end{equation}
We also use the same notations $f(t,u,v)=a(t)uv$, $g(t,u)=\left(b(t)-\frac{c(t)}{1+u^n}\right)u^2$ 
as in  Example~\ref{example5a}.
This leads to 
$f(t,K,K)-g(t,K)\geq (a(t)-b(t))K^2\geq \alpha K^2\geq \alpha K_0>0$ for $K\geq K_0$.
Thus a solution of (\ref{10}) with any positive initial function $\varphi(t) \geq K_0$ is unbounded by 
Theorem~\ref{theorem9}, for any $K_0>0$.

Consider a modification of (\ref{10})
$$
\dot{x}(t)=a(t)x^{\beta}(t-h)x^{\gamma}(t-g)-\left(b(t)-\frac{c(t)}{1+x^n(t)}\right)x^2(t),
$$
where $a(t)\geq 0, b(t)\geq c(t)\geq 0, b(t)-c(t)\geq b_0 >0 $, $a,b,c$ are Lebesgue measurable bounded functions,
$h>0, g>0$, $\beta,\gamma,n\geq 0$, $\beta+\gamma \geq 2$, and (\ref{ex_unb1}) is satisfied.
By the same calculations as before any solution with the initial 
function $\varphi(t) \geq K_0>1$ is unbounded. 
\end{example}

\begin{example}
\label{linear}
Consider the linear equation with several delays
$$
\dot{x}(t)=\sum_{k=1}^m a_k(t) x(h_k(t))-b(t)x(t),
$$
where $a_k(t)\geq 0$, $b(t)\geq \beta>0,$ $a_k,b:[0,\infty) \to [0,\infty)$ and $h_k(t)\leq t$ are Lebesgue measurable bounded functions.
Assume that
$$
\liminf_{t\rightarrow\infty}\frac{\sum_{k=1}^m a_k(t)}{b(t)}>1.
$$
Denote $f_k(t,u)=a_k(t)u$, $g(t,u)=b(t)u$. Then there exist $t_0\geq 0$ and $\alpha>1$
such that $\displaystyle \frac{\sum_{k=1}^m a_k(t)}{b(t)}\geq \alpha$ for $t\geq t_0$,
and $\sum_{k=1}^m a_k(t) -b(t)\geq (1-\alpha)b(t)\geq (1-\alpha)\beta$ for $t\geq t_0$.
Hence 
$$
\sum_{k=1}^m f_k(t,K)-g(t,K)=\left(\sum_{k=1}^m a_k(t) -b(t)\right) K\geq (1-\alpha)\beta K
$$
for any $K\geq K_0>0.$
Then a solution with any initial function, with a positive lower bound, is unbounded by 
Theorem~\ref{theorem9}.
\end{example}

\section{Discussion}

In the present paper, we have studied existence of global positive solutions for nonlinear equation~(\ref{3}) with
several delays, as well as boundedness and persistence of these solutions. The results were applied, for example, to the 
Mackey-Glass equation of population dynamics with non-monotone feedback \cite{MG3}. However, they can also be applied to
some other models, including the Nicholson's blowflies equation with two delays
$$
\dot{x}(t)= P(t)x(h(t)) e^{-x(g(t))} - \delta(t) x(t)
$$
which in the case when variable delays are equal $h(t)=g(t)$ was studied, for example, in
\cite{review_nichol,Kinz2006,Zhuk2012}.


Permanence of solutions of equations of type (\ref{3}) was recently explored in \cite{Faria} and \cite{GHM}. 
Compared to \cite{Faria,GHM}, we consider a more general model: in particular, it is not always assumed that 
$f$ is increasing in all $u$-arguments, as well as continuity in $t$. Also, (H3) in \cite[p. 86]{Faria}
is a special case of  conditions of the present paper. On the other hand, in \cite{Faria,GHM}, solution bounds 
are obtained and more advanced asymptotic properties, such as stability, are discussed.

Equation~(\ref{3}) is a special  case of the equation with a distributed delay
\begin{equation}\label{distr}
\dot{x}(t)=\sum_{k=1}^m f_k\left(t, \int_{h_1(t)}^t x(s)d_s R_1(t,s),\dots,\int_{h_l(t)}^t x(s)d_s R_l(t,s)\right)-g(t,x(t)),
\end{equation}
while the integro-differential equation
\begin{equation}\label{int}
\dot{x}(t)=\sum_{k=1}^m f_k\left(t, \int_{h_1(t)}^t K_1(t,s)x(s)ds ,\dots,\int_{h_l(t)}^t K_l(t,s)x(s)ds \right)-g(t,x(t)),
\end{equation}
is another particular  case of Eq.~(\ref{distr}). All conditions for boundedness,
persistence, permanence and existence of unbounded solutions, obtained here for (\ref{3}) can be extended
to (\ref{distr}) and (\ref{int}), using the ideas of the proofs of the present paper. 

Equations with several delays involved in a nonlinear function is a challenging object with properties quite 
different from the case when these delays coincide, and we have presented several examples to outline this difference. 
However, so far only existence of a positive global solution, persistence and boundedness have been explored.  
It is interesting to investigate other qualitative properties for Eqs.~(\ref{3}), (\ref{distr}) and 
(\ref{int}), such as oscillation, stability and existence of periodic or almost periodic solutions.

One of the main results in this paper is Theorem~\ref{theorem8a}, where we obtain a priori estimations of solutions
for equation~(\ref{3}). Such estimations were used in \cite{BBI2} to obtain global asymptotic stability results for various types of
nonlinear delay differential equations. We expect that this technique can be applied to obtain explicit global stability 
results for Eqs.~(\ref{4}) and (\ref{5}). 

We conclude this discussion by noticing that there are many equations which have a different form than (\ref{3}),
for example, the equation
$
\dot{x}(t)=f(t,x(h(t)))-g(t,x(r(t)))
$
with the delay in the negative term, and the logistic-type equation
\begin{equation}
\label{2logistic}
\dot{x}(t)=r(t)x(h(t))[1-x(g(t))].
\end{equation}
Compared to \eqref{2logistic}, the Hutchinson equation, which is a standard delay-type logistic equation, has $h(t) \equiv t$.
Another delay versions of the logistic equation were considered in \cite{AWW,BBDS}.

However, it is known that Eq.~(\ref{2logistic}) does not even necessarily have a global positive solution. 
It would be interesting to develop a technique to study such new classes of delay differential equations
including \eqref{2logistic}.


\section{Acknowledgments}

L. Berezansky was partially supported by Israeli Ministry of Absorption,
E. Braverman was partially supported by the NSERC research grant RGPIN-2015-05976.
The authors are grateful to the reviewers whose thoughtful comments significantly 
contributed to the presentation of the results of the paper.

\end{document}